\documentclass[reqno,a4paper,draft]{amsart}

\usepackage{enumitem}
\setenumerate{label=\textnormal{(\arabic*)}}

\usepackage{amsmath,amssymb,dsfont,verbatim,bm,array,mathtools}
\usepackage[latin1]{inputenc}
\usepackage{booktabs}
\usepackage[raggedright]{titlesec}
\usepackage{mathtools}

\titleformat{\chapter}[display]
{\normalfont\huge\bfseries}{\chaptertitlename\\thechapter}{20pt}{\Huge}
\titleformat{\section}
{\normalfont\Large\bfseries\center}{\thesection}{1em}{}
\titleformat{\subsection}
{\normalfont\large\bfseries}{\thesubsection}{1em}{}
\titleformat{\subsubsection}[runin]
{\normalfont\normalsize\bfseries}{\thesubsubsection}{1em}{}
\titleformat{\paragraph}[runin]
{\normalfont\normalsize\bfseries}{\theparagraph}{1em}{}
\titleformat{\subparagraph}[runin]
{\normalfont\normalsize\bfseries}{\thesubparagraph}{1em}{}
\titlespacing*{\chapter} {0pt}{50pt}{40pt}
\titlespacing*{\section} {0pt}{3.5ex plus 1ex minus .2ex}{2.3ex plus .2ex}
\titlespacing*{\subsection} {0pt}{3.25ex plus 1ex minus .2ex}{1.5ex plus .2ex}
\titlespacing*{\subsubsection}{0pt}{3.25ex plus 1ex minus .2ex}{1.5ex plus .2ex}
\titlespacing*{\paragraph} {0pt}{3.25ex plus 1ex minus .2ex}{1em}
\titlespacing*{\subparagraph} {\parindent}{3.25ex plus 1ex minus .2ex}{1em}

\input xypic
\xyoption{all}


\newtheorem{theorem}{Theorem}[section]
\newtheorem{lemma}[theorem]{Lemma}
\newtheorem{proposition}[theorem]{Proposition}
\newtheorem{corollary}[theorem]{Corollary}
\newtheorem{conjecture}[theorem]{Conjecture}

\theoremstyle{definition}
\newtheorem{definition}[theorem]{Definition}

\newtheorem{example}[theorem]{Example}

\theoremstyle{remark}
\newtheorem{remark}[theorem]{Remark}

\DeclareMathOperator{\Jac}{Jac}
\DeclareMathOperator{\Img}{Img}

\begin{document}
\title{Observations on the two dimensional Jacobian Conjecture}
\author{Vered Moskowicz}
\address{Department of Mathematics, Bar-Ilan University, Ramat-Gan 52900, Israel.}
\email{vered.moskowicz@gmail.com}
\thanks{The author was partially supported by an Israel-US BSF grant 2010/149}

\begin{abstract}
The two dimensional Jacobian Conjecture says that a morphism 
$f:\mathbb{C}[x,y]\to \mathbb{C}[x,y]$
having an invertible Jacobian, is invertible.
We show that a morphism $f$ having an invertible Jacobian is invertible,
in each of the following two special cases: 
The degree of $f(x)$ is $\leq 2$;
The $(0,1)$-degrees or $(1,0)$-degrees
of all monomials in $f(x)$ are of
the same parity.

In each case there is no restriction on the degree of $f(y)$
nor on the parity of the $(0,1)$-degrees or $(1,0)$-degrees of its monomials.
\end{abstract}

\maketitle                  

\section{Introduction}
The $n$ dimensional Jacobian Conjecture says that a morphism 
$f:\mathbb{C}[x_1,\ldots,x_n]\to \mathbb{C}[x_1\ldots,x_n]$
having an invertible Jacobian, is invertible; 
see ~\cite{keller}.

Wang ~\cite[Theorem 61]{wang sep} showed the following:
``Let $D$ be a UFD with $2 \neq 0$, and let
$D[y_1,\ldots,y_n] \subseteq D[x_1,\ldots,x_n]$
be a separable ring extension
($D[y_1,\ldots,y_n]$ and $D[x_1\ldots,x_n]$
are two polynomial rings, each in $n$ variables).
If the degree of each $y_i$,
considered as a polynomial in $x_1,\ldots,x_n$,
is $\leq 2$,
then $D[y_1,\ldots,y_n] = D[x_1\ldots,x_n]$".
Wang also showed in ~\cite[Theorem 38]{wang sep} that if $R$ is a ring and
$f: R[x_1,\ldots,x_n] \to R[x_1,\ldots,x_n]$
is a morphism
that satisfies 
$\Jac(f(x_1),\ldots,f(x_n)) \in R[x_1,\ldots,x_n]^*$,
then
$R[f(x_1),\ldots,f(x_n)] \subseteq R[x_1,\ldots,x_n]$
is separable. 
(~\cite[Theorem 38]{wang sep} shows that the converse is also true,
namely, separability implies invertibility of the Jacobian).
Combining the two yields Wang's theorem 
~\cite[Theorem 62]{wang sep} ~\cite[Proposition 4.3.1]{essen}:
``Let $D$ be a UFD with $2 \neq 0$,
and let 
$f: D[x_1,\ldots,x_n] \to D[x_1,\ldots,x_n]$
be a morphism that satisfies
$\Jac(f(x_1),\ldots,f(x_n)) \in D^*$.
If the degree of each $f(x_i)$ is $\leq 2$,
then $D[f(x_1),\ldots,f(x_n)] = D[x_1\ldots,x_n]$,
namely $f$ is invertible". 

Here we will mostly deal with $n=2$,
$k$ a field of characteristic zero
and 
$f: k[x,y] \to k[x,y]$ 
a morphism that satisfies
$\Jac(f(x),f(y)) \in k^*$;
for convenience denote
$p:=f(x)$ and $q:=f(y)$.
In Theorem \ref{1} $k=\mathbb{C}$
and in Theorem \ref{2} $k=\mathbb{C}$
or $k=\mathbb{R}$.

Wang's theorem for $n=2$ and $D=\mathbb{C}$ says that if 
$p$ is of degree $\leq 2$ and $q$ is of degree $\leq 2$,
then $f$ is invertible.
In Theorem \ref{1}, we show that if 
$p$ is of degree $\leq 2$
(and $q$ can have any degree),
then $f$ is invertible.
Clearly this is a generalization of Wang's theorem
when $n=2$ and $D=\mathbb{C}$.
Notice two great things in Wang's theorem: First, it is valid for all $n$, 
and second, $D$ is a UFD with $2 \neq 0$.
When $n=2$ and $D$ is a field of characteristic zero,
it is not difficult to show that a morphism $f$ of Wang's theorem
is invertible, see Proposition \ref{prop wang 2}.
\begin{remark}\label{other nice conditions}
There exist other nice conditions when $n=2$ and 
$D$ is a field of characteristic zero,
each implies that $f$ is invertible: 
\begin{itemize}
\item [(1)] The degree of $p$ is $\leq 100$ 
and the degree of $q$ is $\leq 100$;
Moh ~\cite{moh} ~\cite[Theorem 10.2.30]{essen}.
\item [(2)] The degree of $p$ or the degree of $q$ is a prime number;
Magnus ~\cite{magnus} ~\cite[Corollary 10.2.25]{essen}. 
\item [(3)] The greatest common divisor of the degrees of $p$
and $q$ is $\leq 2$; Nakai and Baba ~\cite{nakai baba}.
\item [(4)] The greatest common divisor of the degrees of $p$
and $q$ is $\leq 8$ or a prime number;
Appelgate, Onishi, Nagata
~\cite{nagata1} ~\cite{nagata2} ~\cite[Theorem 10.2.26]{essen}.
\end{itemize}
\end{remark}

Given
$f: k[x_1,\ldots,x_n] \to k[x_1\ldots,x_n]$,
$f(x_i)$ is a polynomial in $x_1,\ldots,x_n$,
so we can consider its $x_j$-degree, 
which is also called the 
$(0,\ldots,1,\ldots,0)$-degree of $f(x_i)$, 
where $1$ is in the $j$'th place.
(The $(1,\ldots,1)$-degree of $f(x_i)$
is usually called the degree of $f(x_i)$).
In Theorem \ref{2},
which does not generalize a known result, but may be of some interest,
we show that if the $(0,1)$-degrees or $(1,0)$-degrees 
of all monomials of $p$ 
have the same parity, then $f$ is invertible.

\section{Preliminaries}
Our proofs of Theorem \ref{1} and Theorem \ref{2}
rely on results found in ~\cite{vered valqui} and ~\cite{vered}.
For the convenience of the reader, we now bring the specific results 
which are needed in the proofs of Theorem \ref{1} and Theorem \ref{2}.

Let $k$ be a field of characteristic zero.
Let $\alpha$ be the exchange involution on $k[x,y]$,
namely $\alpha$ is the morphism of order $2$ (hence invertible)
defined by
$\alpha(x)=y$ and $\alpha(y)=x$.
A morphism $g: k[x,y] \to k[x,y]$
is an $\alpha$-morphism if
$g \alpha= \alpha g$.
If $\sigma$ is any involution on $k[x,y]$,
then a morphism $g: k[x,y] \to k[x,y]$
is a $\sigma$-morphism if $g \sigma= \sigma g$.
In those definitions there is no demand on the Jacobian of $g$; 
however, in the results, we will always assume that a given 
$\sigma$-morphism has an invertible Jacobian.
We will also work with the following involutions:
$\beta(x)= x, \beta(y)= -y$,
$\gamma(x)= -x, \gamma(y)= y$,
and 
$\epsilon(x)= -x, \epsilon(y)= -y$.
Notice that $\alpha$, $\beta$ and $\gamma$ 
are in the same conjugacy class:
$\alpha \beta \alpha= \gamma$, 
and $g^{-1} \alpha g= \beta$ where 
$g(x)=(1/2)(x+y), g(y)=y-x$,
$g^{-1}(x)=x-(1/2)y, g^{-1}(y)=x+(1/2)y$.
More generally, given two involutions on $k[x,y]$, $\sigma$ and $\tau$,
a morphism $g: k[x,y] \to k[x,y]$ is
a $\sigma,\tau$-morphism if $g \sigma = \tau g$.

\begin{proposition}\label{two conj classes}
There exist two conjugacy classes of involutions on 
$k[x,y]$ in the group of invertible morphisms
(= automorphisms) of $k[x,y]$: 
\begin{itemize}
\item A class which consists of all involutions having Jacobian $-1$,
denote it by $C_{-1}$.
\item A class which consists of all involutions having Jacobian $1$,
denote it by $C_{1}$.
\end{itemize}
\end{proposition}

Denote $G_1$= affine automorphisms, $G_2$= de Jonquieres automorphisms.
Recall Jung-van der Kulk automorphism theorem which says that
the group of automorphisms 
of $k[x,y]$, $k$ is any field, is the amalgamated free product of
$G_1$ and $G_2$ over their intersection,
see ~\cite[Theorem 5.1.11]{essen}, ~\cite{jung} and ~\cite{kulk}.

\begin{proof}[Sketch of proof]
Two trivial remarks:
\begin{itemize}
\item Two conjugate invertible morphisms (in particular, two conjugate involutions) 
must have the same Jacobian. So all the members of a given 
conjugacy class have the same Jacobian.
\item Given an involution $\delta$, the Jacobian of $\delta$ must equal $1$ or $-1$
(The Jacobian of $\delta$ equals its inverse, since $\delta$ equals its inverse).
\end{itemize}
Apriori, it may happen that there exist two or more conjugacy classes of involutions 
having Jacobian $-1$, and there exist two or more conjugacy classes
of involutions having Jacobian $1$.
However, the following arguments show that there exists only one conjugacy class 
of involutions having Jacobian $-1$, denote it $C_{-1}$,
and there exists only one conjugacy class of involutions having Jacobian $1$,
call it $C_{1}$.
{}From direct calculations we get the following two facts:
\begin{itemize}
\item In $G_1$ there exist exactly two conjugacy classes of involutions, 
that of $\alpha$ (or $\beta$ or $\gamma$) and that of $\epsilon$.
\item In $G_2$ there exist exactly two conjugacy classes of involutions, 
that of $a$ and that of $b$,
where $a(x)= -x-y^2, a(y)=y$ and $b(x)= -x-y^2, b(y)= -y$.
(We obtained $6$ general forms of involutions of $G_2$,
and then checked that some are conjugate to others).
\end{itemize}
A direct calculation shows that: 
\begin{itemize}
\item $a$ and $\beta$ are conjugate: $\beta= h^{-1} a h$,
where $h(x)= -y, h(y)= -x-(1/2)y^2$.
\item $b$ and $\epsilon$ are conjugate: $\epsilon= g^{-1} b g$,
where $g(x)=y, g(y)= -x-(1/2)y^2$.
\end{itemize}
We now explain why in the group of automorphisms of $k[x,y]$ 
(and not only in $G_1 \cup G_2$)
there exist exactly two conjugacy classes of involutions: 
That of $\alpha$ and that of $\epsilon$.
 
The explanation is quite easy thanks to J. Bell ~\cite{bell} 
who told us about Serre's theorem ~\cite[page 6, Corollary 1]{trees}:
``Every element of $G$ of finite order is conjugate to an element of one of the $G_i$",
where $G$ is the amalgamated free product of $G_1$ and $G_2$ over their intersection.
Here $G$ is the group of automorphisms of $k[x,y]$ which is known 
to be the amalgamated free product of $G_1$ and $G_2$ over their intersection
(Jung-van der Kulk).
Indeed, let $\iota$ be any involution on $k[x,y]$. $\iota$ is of order $2$, 
so from Serre's theorem 
$\iota$ must be conjugate to an element, call it $e$, of $G_1$ or of $G_2$.
Trivially, $e$ is also of order $2$ (it is clear that two conjugate elements, 
each of some finite order, must have the same order), namely $e$ is an involution.
$e$ is an involution of $G_1$ or of $G_2$, therefore $e$ is conjugate to 
$\alpha$ or $\epsilon$, so $\iota$ is conjugate to $\alpha$ or $\epsilon$.
\end{proof}

\begin{remark}\label{remark gamma delta}
Let $g: k[x,y] \to k[x,y]$ be a morphism
such that 
$\Jac(g(x),g(y)) \in k^*$.
If $g$ is a $\sigma,\tau$-morphism, then
$\sigma,\tau \in C_{-1}$ or
$\sigma,\tau \in C_{1}$.
Indeed, $g \sigma= \tau g$ combined with 
$\Jac(g(x),g(y)) \in k^*$ implies that
the Jacobian of $\sigma$ equals the Jacobian of $\tau$,
so both have Jacobian $-1$ or both have Jacobian $1$.
\end{remark}
A morphism $g: k[x,y] \to k[x,y]$ is invertible
if there exists a morphism
$h: k[x,y] \to k[x,y]$ such that
$gh=hg=1$, where $1$ is the identity morphism.
Given a morphism $g: k[x,y] \to k[x,y]$
(with no restriction on its Jacobian)
denote $\Img(g):=k[g(x),g(y)]$.
$\Img(g)$ is a sub-algebra of $k[x,y]$.
A morphism $g$ of $k[x,y]$ 
is invertible if and only if $\Img(g)= k[x,y]$,
see ~\cite[Lemma 1 in the Introduction; Definition 1.1.5]{essen} and 
~\cite[page 343]{cohn}. 

The connection between involutions and the Jacobian Conjecture begins in 
the following conjecture:

\begin{conjecture}[The $\alpha$ Jacobian Conjecture]\label{alpha conj}
Let $f: k[x,y] \to k[x,y]$ 
be an $\alpha$-morphism that satisfies
$\Jac(f(x),f(y)) \in k^*$.
Then $f$ is invertible.
\end{conjecture}

An analogous conjecture for the first Weyl algebra, 
``the $\alpha$ Dixmier Conjecture",
was made before the $\alpha$ Jacobian Conjecture and was almost proved, 
see ~\cite{vered before valqui}.
More specifically, we first dealt with the first Weyl algebra and found a family of
$\alpha$-endomorphisms which is easily seen to be a family of $\alpha$-automorphisms,
see ~\cite[Proposition 2.8]{vered before valqui},
but we were not able to show that this family includes all 
$\alpha$-endomorphisms.
Another result about the form of an 
$\alpha$-endomorphism is ~\cite[Lemma 2.6]{vered before valqui}.
An almost a proof for the $\alpha$ Dixmier Conjecture is ~\cite[Theorem 2.9]{vered before valqui}.
After reading  ~\cite{vered before valqui}, C. Valqui suggested a nice proof for 
the $\alpha$ Jacobian Conjecture,
which appears in ~\cite[Proposition 4.1]{vered valqui}.

\begin{theorem}[The $\alpha$ Jacobian Conjecture is true]\label{alpha true}
Let $f: k[x,y] \to k[x,y]$ 
be an $\alpha$-morphism that satisfies
$\Jac(f(x),f(y)) \in k^*$.
Then $f$ is invertible.
\end{theorem}

\begin{proof}
See ~\cite[Proposition 4.1]{vered valqui}.

Notice that in ~\cite[Proposition 4.1]{vered valqui} 
$\Jac(f(x),f(y)) =1$; it is clear that if
$\Jac(f(x),f(y)) = a \in k^*$ (with $a$ not necessarily equals $1$),
we can consider $\tilde{f}(x):= a^{-1}f(x)$, 
$\tilde{f}(y):= f(y)$. Then
$\Jac(\tilde{f}(x),\tilde{f}(y)) = 1$,
hence ~\cite[Proposition 4.1]{vered valqui} implies that 
$\tilde{f}$ is invertible, and then $f$ is invertible,
because 
$k[f(x),f(y)]=k[\tilde{f}(x),\tilde{f}(y)] = k[x,y]$. 
\end{proof}

Then it is immediate to get:

\begin{corollary}\label{cor alpha true}
Let $f: k[x,y] \to k[x,y]$ 
be a $\sigma,\tau$-morphism,
$\sigma,\tau \in C_{-1}$,
that satisfies
$\Jac(f(x),f(y)) \in k^*$.
Then $f$ is invertible.
\end{corollary}
 
\begin{proof}
$\sigma,\tau \in C_{-1}$, so there exist invertible morphisms 
$u$ and $v$ such that 
$\sigma = u^{-1} \alpha u$ and
$\tau = v^{-1} \alpha v$.
$f$ is a $\sigma,\tau$-morphism:
$f \sigma = \tau f$.
So
$f u^{-1} \alpha u = v^{-1} \alpha v f$,
then
$(v f u^{-1}) \alpha = \alpha (v f u^{-1})$,
namely,
$v f u^{-1}$ is an $\alpha$-morphism.
{}From Theorem \ref{alpha true} ~\cite[Proposition 4.1]{vered valqui},
we get that $v f u^{-1}$ is invertible.
Indeed, Theorem \ref{alpha true} can be applied here, since from the Chain Rule, 
the fact that any invertible morphism has an invertible Jacobian, 
and the assumption that $f$ has an invertible Jacobian, we obtain:
$\Jac((v f u^{-1})(x), (v f u^{-1})(y)) \in k^*$.
Clearly, invertibility of $v f u^{-1}$ implies invertibility of $f$.
\end{proof}

Similarly to the $\alpha$ Jacobian Conjecture \ref{alpha conj} we have:

\begin{conjecture}[The $\epsilon$ Jacobian Conjecture]\label{epsilon conj}
Let $f: k[x,y] \to k[x,y]$ 
be an $\epsilon$-morphism that satisfies
$\Jac(f(x),f(y)) \in k^*$.
Then $f$ is invertible.
\end{conjecture}

We suspect that the $\epsilon$ Jacobian Conjecture is also true,
though we have not yet proved it.
A positive answer to the $\epsilon$ Jacobian Conjecture will immediately imply
that if $f: k[x,y] \to k[x,y]$ 
is a $\sigma,\tau$-morphism,
$\sigma,\tau \in C_{1}$,
that satisfies
$\Jac(f(x),f(y)) \in k^*$,
then $f$ is invertible;
same proof as that of Corollary \ref{cor alpha true}
(apply the positive answer to the 
$\epsilon$ Jacobian Conjecture instead of 
the positive answer to the $\alpha$ Jacobian Conjecture).
Whether the $\epsilon$ Jacobian Conjecture is true or not, 
thanks to the positive answer to the $\alpha$ Jacobian Conjecture, 
Theorem \ref{alpha true} ~\cite[Proposition 4.1]{vered valqui},
we were able to obtain additional results, 
some of which are brought in ~\cite{vered}; 
we now bring only those results relevant for the proofs of
Theorem \ref{1} and Theorem \ref{2}.

$k$ continues to denote a field of characteristic zero.
For a morphism $g: k[x,y] \to k[x,y]$
that satisfies
$\Jac(g(x),g(y)) \in k^*$, 
$\Img(g)$ is isomorphic to $k[x,y]$
(since if $\Jac(g(x),g(y)) \in k^*$, 
then, by ~\cite{rowen}, $g(x)$ and $g(y)$ are 
algebraically independent over $k$).

\begin{definition}[The $\alpha$ restriction condition]
We say that a morphism $g: k[x,y] \to k[x,y]$ 
satisfies the $\alpha$ restriction condition if
$\alpha(g(x)) \in \Img(g)$ and $\alpha(g(y)) \in \Img(g)$.
Equivalently, we say that $g$ satisfies the $\alpha$ 
restriction condition if the exchange involution
$\alpha$ on $k[x,y]$ when restricted to $\Img(g)$ 
is an involution on $\Img(g)$.
\end{definition}

We do not know if a morphism $f: k[x,y] \to k[x,y]$
that satisfies 
$\Jac(f(x),f(y)) \in k^*$
necessarily satisfies the $\alpha$ restriction condition,
but if it does, then it is invertible:

\begin{theorem}[The $\alpha$ restriction theorem]\label{alpha res thm}
Let $f$ be a morphism of $k[x,y]$ that satisfies
$\Jac(f(x),f(y)) = 1$.
$f$ satisfies the $\alpha$ restriction condition
$\Longleftrightarrow$ $f$ is invertible.
\end{theorem}

An immediate (and trivial) corollary to Theorem \ref{alpha res thm}
is as follows: Let $f$ be a morphism of $k[x,y]$ 
that satisfies
$\Jac(f(x),f(y)) \in k^*$.
$f$ satisfies the $\alpha$ restriction condition
$\Longleftrightarrow$ $f$ is invertible.

\begin{proof}
$\Longrightarrow$: 
$f$ satisfies the $\alpha$ restriction condition,
so $\alpha$ restricted to $\Img(f)$ is an involution on $\Img(f)$, 
$\Img(f):=k[f(x),f(y)]$.
Denote the restriction of $\alpha$ to $\Img(f)$ by $\alpha_0$.
The Jacobian of $\alpha_0$ is $-1$:
$\Jac(\alpha_0(f(x)),\alpha_0(f(y)))$
$= \Jac(\alpha(f(x)),\alpha(f(y)))$
$= \Jac((\alpha f)(x),(\alpha f)(y))= -1$
(by the Chain Rule).
Denote the involution on $\Img(f)$ which exchanges 
$f(x)$ and $f(y)$ by $\rho_0$.
Obviously, $\rho_0$ has Jacobian $-1$
(since we assumed $\Jac(f(x),f(y)) = 1$).
$\Img(f)$ is isomorphic to $k[x,y]$, hence from Proposition \ref{two conj classes} 
there exist two conjugacy classes of involutions on $\Img(f)$: 
Involutions having Jacobian $-1$ and involutions having Jacobian $1$.
$\alpha_0$ and $\rho_0$ both have Jacobian $-1$,
hence both belong to the same conjugacy class,
so there exists an invertible morphism $h_0: \Img(f) \to \Img(f)$
such that 
$\rho_0= h_0^{-1} \alpha_0 h_0$.
Therefore,
$(f \alpha)(x)= f(\alpha(x))= f(y)$
$=\rho_0(f(x))= (h_0^{-1} \alpha_0 h_0)(f(x))$
$=(h_0^{-1} \alpha_0 h_0 f)(x)$,
and,
$(f \alpha)(y)= f(\alpha(y))= f(x)$ 
$=\rho_0(f(y))= (h_0^{-1} \alpha_0 h_0)(f(y))$
$(h_0^{-1} \alpha_0 h_0 f)(y)$.
Therefore, 
$f \alpha= h_0^{-1} \alpha_0 h_0 f$.
Then, $h_0 f \alpha= \alpha_0 h_0 f$,
so $h_0 f \alpha= \alpha h_0 f$,
namely $h_0 f$ is an $\alpha$-morphism of $k[x,y]$.
Since the Jacobian of $h_0(f(x)),h_0(f(y))$
with respect to $f(x),f(y)$
is a non-zero scalar 
($h_0$ is an invertible morphism of $\Img(f)$)
and the Jacobian of $f(x),f(y)$ with respect to $x,y$
equals $1$ (by assumption),
we get that the Jacobian of 
$(h_0 f)(x),(h_0 f)(y)$ 
with respect to $x,y$
is a non-zero scalar (by the Chain Rule).
By Theorem \ref{alpha true} ~\cite[Proposition 4.1]{vered valqui}
$h_0 f$ is an invertible morphism of $k[x,y]$,
namely, 
$k[(h_0f)(x),(h_0f)(y)]= 
k[x,y]$.
Then we have:
$x= \sum a_{ij} ((h_0f)(x))^i ((h_0f)(y))^j$
$= \sum a_{ij} (h_0(f(x)))^i (h_0(f(y)))^j$

$= h_0 (\sum a_{ij} f(x)^i f(y)^j)$,
and
$y= \sum b_{ij} ((h_0f)(x))^i ((h_0f)(y))^j$

$= \sum b_{ij} (h_0(f(x)))^i (h_0(f(y)))^j$
$= h_0 (\sum b_{ij} f(x)^i f(y)^j)$,
where $a_{ij}, b_{ij} \in k$.
This shows that $x,y \in \Img(h_0)= \Img(f)$, 
so $\Img(f) =k[x,y]$, and $f$ is invertible.

$\Longleftarrow$: $f$ is invertible,
so $\Img(f)= k[x,y]$.
Therefore, $f$ satisfies the $\alpha$ restriction condition, since trivially
$\alpha(f(x)) \in k[x,y]= \Img(f)$ and $\alpha(f(y)) \in k[x,y]= \Img(f)$.
\end{proof}
Theorem \ref{p symmetric alpha} and its corollary \ref{p symmetric delta}
will be applied in the proofs of Theorem \ref{1}
and Theorem \ref{2}; the proof of Theorem \ref{p symmetric alpha} 
relies on a theorem of Cheng-Mckay-Wang ~\cite[Theorem 1]{wang younger}
and on the $\alpha$ restriction theorem \ref{alpha res thm}
(Theorem \ref{alpha res thm} relies on Theorem 
\ref{alpha true} ~\cite[Proposition 4.1]{vered valqui},
as we have just seen).

Before bringing Theorem \ref{p symmetric alpha}, we wish to discuss
Cheng-Mckay-Wang's theorem ~\cite[Theorem 1]{wang younger} 
which says the following:
``Let $L$ be the field of complex numbers. Assume $A,B \in L[x,y]$ satisfy
$\Jac(A,B) \in L^*$.
If $R \in L[x,y]$ satisfies $\Jac(A,R)= 0$, then $R \in L[A]$".
In other words, C-M-W's theorem ~\cite[Theorem 1]{wang younger} says that 
``the centralizer with respect to the Jacobian"
of an element $A \in \mathbb{C}[x,y]$ which has a Jacobian mate,
equals $\mathbb{C}[A]$.
By definition, a Jacobian mate of an element $A \in k[x,y]$ is
an element $B \in k[x,y]$ such that $\Jac(A,B) \in k^*$.
Its analogous result in the first Weyl algebra over any characteristic zero field,
not necessarily the field of complex numbers, 
can be found in ~\cite[Theorem 2.11]{ggv};
instead of the Jacobian take the commutator.

The following Lemma shows that C-M-W's theorem over $L=\mathbb{C}$ implies
C-M-W's theorem over $L=\mathbb{R}$:

\begin{lemma}\label{cmw}
Assume $A,B \in \mathbb{R}[x,y]$ satisfy
$\Jac(A,B) \in \mathbb{R}^*$.
If $R \in \mathbb{R}[x,y]$ satisfies $\Jac(A,R)= 0$, 
then $R \in \mathbb{R}[A]$.
\end{lemma}

\begin{proof}
$A,B \in \mathbb{R}[x,y] \subset \mathbb{C}[x,y]$ satisfy
$\Jac(A,B) \in \mathbb{R}^* \subset \mathbb{C}^*$,
so C-M-W's theorem implies that $R \in \mathbb{C}[A]$
(of course, $R \in \mathbb{R}[x,y] \subset \mathbb{C}[x,y]$).
Hence $R= \sum_{u=0}^{t} c_u A^u$ for some $c_u \in \mathbb{C}$.
We claim that actually $c_u \in \mathbb{R}$:
On the one hand, $R \in \mathbb{R}[x,y]$, so we can write
$R= \sum_{i=0}^{k} \sum_{j=0}^{l} r_{ij} x^i y^j$ for some 
$r_{ij} \in \mathbb{R}$.
On the other hand, $R= \sum_{u=0}^{t} c_u A^u$ 
for some $c_u \in \mathbb{C}$.
Therefore $\sum^{k} \sum^{l} r_{ij} x^i y^j = \sum^{t} c_u A^u$.
We can write
$\mathbb{R}[x,y] \ni A= 
\sum_{v=0}^{n} \sum_{w=0}^{m} a_{vw} x^v y^w$,
where $a_{vw} \in \mathbb{R}$.
So $\sum^{k} \sum^{l} r_{ij} x^i y^j =
\sum_{u=0}^{t} c_u (\sum^{n} \sum^{m} a_{vw} x^v y^w)^{u}$.
The leading coefficient on the right is $c_t a_{nm}^t$,
and the leading coefficient on the left is $r_{kl}$.
Hence $r_{kl}=c_t a_{nm}^t$, so $c_t= r_{kl}/ (a_{nm}^t) \in \mathbb{R}$,
since $r_{kl} \in \mathbb{R}$ and $a_{nm} \in \mathbb{R}$.
Now consider
$\mathbb{R}[x,y] \ni \sum \sum r_{ij} x^i y^j - c_t A^t =
\sum_{u=0}^{t-1} c_u (\sum \sum a_{vw} x^v y^w)^u$,
and similarly obtain that $c_{t-1} \in \mathbb{R}$.
Indeed, the leading coefficient on the right is $c_{t-1} a_{nm}^{t-1}$,
and the leading coefficient on the left is some $d \in \mathbb{R}$.
Hence $c_{t-1} = d/ (a_{nm}^{t-1}) \in \mathbb{R}$.
Continuing this way we obtain that $c_0,c_1\ldots,c_{t-1},c_t$
are all in $\mathbb{R}$.
Therefore, $R= \sum_{u=0}^{t} c_u A^u \in \mathbb{R}[A]$.
\end{proof}
As for the validity of C-M-W's theorem 
over fields other than $\mathbb{C}$ or $\mathbb{R}$,
perhaps it is still valid over any field of characteristic zero,
as the two following lemmas may show. 
However, the first lemma \ref{lemma alg clo char 0} 
does not have a complete proof, 
only a sketch of proof, 
so it may happen that the first lemma is not true.

\begin{lemma}\label{lemma alg clo char 0}
Let $L$ be an algebraically closed field of characteristic zero.
Assume $A,B \in L[x,y]$ satisfy
$\Jac(A,B) \in L^*$.
If $R \in L[x,y]$ satisfies $\Jac(A,R)= 0$, then $R \in L[A]$.
\end{lemma}

\begin{proof}[Sketch of proof]
The proof of C-M-W's theorem uses ~\cite[Theorem 1]{wang deri}
and ~\cite[Corollary 1.5, p.74]{bass2}; except these two results, it seems 
that the other arguments in C-M-W's proof are valid not only over $\mathbb{C}$
but also over other fields
(we are not sure over which fields the other arguments are valid; 
maybe any commutative rings).
So if one wishes to generalize C-M-W's theorem to an algebraically closed field
of characteristic zero,
it is enough to check that each of ~\cite[Theorem 1]{wang deri}
and ~\cite[Corollary 1.5, p. 74]{bass2} is valid over
an algebraically closed field of characteristic zero.
Indeed, 
in ~\cite[Theorem 1]{wang deri} the base field is any ring,
and in ~\cite[Corollary 1.5, p. 74]{bass2} the base field is 
an algebraically closed field of characteristic zero.
\end{proof}

If we do not have any errors in the sketch of proof of Lemma 
\ref{lemma alg clo char 0}, then similarly to the proof of Lemma \ref{cmw}
(=getting the result over $\mathbb{R}$ from the result over $\mathbb{C}$)
we can obtain the following lemma
(=getting the result over $L$ from the result over $\bar{L}$,
where $L$ is any field of characteristic zero, 
$\bar{L}$ is an algebraic closure of $L$).

\begin{lemma}\label{lemma char 0}
Let $L$ be a field of characteristic zero.
Assume $A,B \in L[x,y]$ satisfy
$\Jac(A,B) \in L^*$.
If $R \in L[x,y]$ satisfies $\Jac(A,R)= 0$, then $R \in L[A]$.
\end{lemma}

\begin{proof}
$A,B \in L[x,y] \subset \bar{L}[x,y]$ satisfy
$\Jac(A,B) \in L^* \subset \bar{L}^*$,
and $R \in L[x,y] \subset \bar{L}[x,y]$
satisfies $\Jac(A,R)= 0$.
{}From Lemma \ref{lemma alg clo char 0} we get that
$R \in \bar{L}[A]$.
Exactly the same arguments as in the proof of Lemma
\ref{cmw} are valid if, instead of working with
$\mathbb{R}$ and $\mathbb{C}$,
we work with $L$ and $\bar{L}$.
\end{proof}
Since in the above discussion there was no precise conclusion whether 
C-M-W's theorem is valid not only over $\mathbb{C}$ and $\mathbb{R}$
but also over any field of characteristic zero
(since Lemma \ref{lemma alg clo char 0} was not fully proved),
we restrict the base field of Theorem 
\ref{p symmetric alpha} to $\mathbb{C}$ or $\mathbb{R}$ 
(because we wish to use C-M-W's theorem in the proof of 
Theorem \ref{p symmetric alpha}).

\begin{theorem}\label{p symmetric alpha}
Let $K$ denote $\mathbb{C}$ or $\mathbb{R}$.
Let $f: K[x,y] \to K[x,y]$ be a morphism that satisfies
$\Jac(f(x),f(y)) \in K^*$.
If one of the following conditions is satisfied,
then $f$ is invertible:
\begin{itemize}
\item [(1)] $f(x)$ is symmetric.
\item [(2)] $f(x)$ is skew-symmetric.
\item [(3)] $f(y)$ is symmetric.
\item [(4)] $f(y)$ is skew-symmetric.
\end{itemize}
Where by symmetric or skew-symmetric we mean symmetric or 
skew-symmetric with respect to $\alpha$.
\end{theorem}

It is impossible to have both $f(x)$ and $f(y)$ symmetric with respect to $\alpha$
or both skew-symmetric with respect to $\alpha$, see ~\cite[Remark 4.8]{vered}.
If it happens that one of $f(x),f(y)$ is symmetric w.r.t. $\alpha$ 
and the other is skew-symmetric w.r.t $\alpha$, then it is quite immediate 
(as long as we know Theorem \ref{alpha true} ~\cite[Proposition 4.1]{vered valqui})
that such $f$ is invertible, see ~\cite[Remark 4.8]{vered}.

\begin{proof}
For convenience, denote $p:= f(x)$, $q:= f(y)$, and $T:=\Img(f)=K[p,q]$

$(1)$: Assume that $p$ is symmetric w.r.t. $\alpha$, 
namely $\alpha(p)= p$.
So $\alpha(p)= p \in T$.
Denote $a:= \Jac(p,q) \in K^*$.
Clearly,
$\Jac(p,\alpha(q))= \Jac(\alpha(p),\alpha(q)) = -a$.
Then,
$\Jac(p,q +\alpha(q))= \Jac(p,q)+ \Jac(p,\alpha(q))= a-a= 0$,
so from C-M-W's theorem (and our Lemma \ref{cmw},
in case $K=\mathbb{R}$) we have
$q + \alpha(q)= H(p)$ where $H(t) \in K[t]$.
So $\alpha(q)= -q + H(p) \in T$.
We have,
$\alpha(p) \in T$ and $\alpha(q) \in T$,
namely, $f$ satisfies the $\alpha$ restriction condition.
Hence the $\alpha$ restriction theorem \ref{alpha res thm} 
(or its immediate and trivial corollary) implies that $f$ is invertible.
$(2)$: Assume that $p$ is skew-symmetric w.r.t. $\alpha$, 
namely $\alpha(p)= -p$.
So $\alpha(p)= -p \in T$.
Denote $a:= \Jac(p,q) \in K^*$.
Clearly,
$\Jac(-p,\alpha(q))= \Jac(\alpha(p),\alpha(q)) = -a$.
Then,
$\Jac(p,q-\alpha(q))= \Jac(p,q)+\Jac(p,-\alpha(q))= a-a= 0$,
so from C-M-W's theorem (and our Lemma \ref{cmw},
in case $K=\mathbb{R}$) we have
$q - \alpha(q)= H(p)$ where $H(t) \in K[t]$.
So $\alpha(q)= q - H(p) \in T$.
We have,
$\alpha(p) \in T$ and $\alpha(q) \in T$,
namely, $f$ satisfies the $\alpha$ restriction condition.
Hence the $\alpha$ restriction theorem \ref{alpha res thm} 
(or its immediate and trivial corollary)
implies that $f$ is invertible.
\end{proof}

In the proof of Theorem \ref{alpha res thm} we have shown that 
for a morphism $f$ having an invertible Jacobian, if the restriction of $\alpha$ to 
$k[f(x),f(y)]$ is an involution on $k[f(x),f(y)]$, then 
from the Chain Rule one sees that it belongs to the 
conjugacy class of involutions on 
$k[f(x),f(y)]$ having Jacobian $-1$.
In the proof of $(1)$ of Theorem \ref{p symmetric alpha} we have found that 
the restriction of $\alpha$ to $K[f(x),f(y)]$ is as follows:
$\alpha(f(x))=f(x)$ and $\alpha(f(y))= -f(y) + H(f(x))$, 
so, without using the Chain Rule, one immediately gets
that the Jacobian of the restricted $\alpha$ equals $-1$;
of course the partial derivatives are with respect to 
$f(x)$ and $f(y)$,
not with respect to $x$ and $y$.
(In Theorem \ref{alpha res thm} $k$ is any field of characteristic zero,
while in Theorem \ref{p symmetric alpha} $K$ is
$\mathbb{C}$ or $\mathbb{R}$).

\begin{corollary}\label{p symmetric delta}
Let $K$ denote $\mathbb{C}$ or $\mathbb{R}$,
and $\delta \in C_{-1}$.
Let $f: K[x,y] \to K[x,y]$ 
be a morphism that satisfies
$\Jac(f(x),f(y)) \in K^*$.
If one of the following conditions is satisfied,
then $f$ is invertible:
\begin{itemize}
\item [(1)] $f(x)$ is symmetric.
\item [(2)] $f(x)$ is skew-symmetric.
\item [(3)] $f(y)$ is symmetric.
\item [(4)] $f(y)$ is skew-symmetric.
\end{itemize}
Where by symmetric or skew-symmetric we mean symmetric or 
skew-symmetric with respect to $\delta$.
\end{corollary}

\begin{proof}
$(1)$: Assume that $f(x)$ is symmetric w.r.t. $\delta$, 
namely $\delta(f(x))= f(x)$.
Since $\delta \in C_{-1}$, there exists an invertible morphism of 
$K[x,y]$, $g$, such that
$\delta = g^{-1} \alpha g$.
Hence $\delta(f(x))= f(x)$ becomes
$(g^{-1} \alpha g)(f(x))= f(x)$,
so 
$\alpha((gf)(x))= (gf)(x)$.
It is clear that the morphism $gf$ has an invertible Jacobian
(since each of $g$ and $f$ has an invertible Jacobian).
Therefore, we can apply Theorem \ref{p symmetric alpha} condition $(1)$ 
to $gf$, and get that $gf$ is invertible.
Then clearly $f$ is invertible.

$(2)$: Assume that $f(x)$ is skew-symmetric w.r.t. $\delta$, 
namely $\delta(f(x))= - f(x)$.
Since $\delta \in C_{-1}$, 
there exists an invertible morphism of 
$K[x,y]$, $g$, such that
$\delta = g^{-1} \alpha g$.
Hence $\delta(f(x))= - f(x)$ becomes
$(g^{-1} \alpha g)(f(x))= - f(x)$,
so 
$\alpha((gf)(x))= - (gf)(x)$.
It is clear that the morphism $gf$ has an invertible Jacobian
(since each of $g$ and $f$ has an invertible Jacobian).
Therefore, we can apply Theorem \ref{p symmetric alpha} condition $(2)$ 
to $gf$, and get that $gf$ is invertible.
Then clearly $f$ is invertible.
\end{proof}

\section{Our two observations}

In theorem \ref{1} we take $\mathbb{C}$ as a base field, while in Theorem \ref{2}
we take $\mathbb{C}$ or $\mathbb{R}$ as a base field; both proofs rely on 
Corollary \ref{p symmetric delta} which is over $\mathbb{C}$ or $\mathbb{R}$,
but in the proof of Theorem \ref{1} we also wish that
square roots of elements of the base field belong to the base field, 
so we dismiss of $\mathbb{R}$.

If it will turn out that Lemma \ref{lemma char 0}
(which relies on Lemma \ref{lemma alg clo char 0}) is true,
then Theorem \ref{p symmetric alpha} and its corollary \ref{p symmetric delta}
are valid over any field of characteristic zero, 
not only over $\mathbb{C}$ or $\mathbb{R}$,
and then Theorem \ref{2} is valid over any field of characteristic zero
and Theorem \ref{1} is valid over any field of characteristic zero
which has the property that square roots of elements of it belong to it
(for example, an algebraically closed field of characteristic zero).

$k$ continues to denote a field of characteristic zero.

\begin{lemma}[Degree 1 case]\label{degree 1 case}
Let $f: k[x,y] \to k[x,y]$ be a morphism 
that satisfies
$\Jac(f(x),f(y)) \in k^*$.
If one of $f(x),f(y)$ is of degree $1$,
then $f$ is invertible.
\end{lemma}

\begin{proof}
Write 
$f(x)= ax+by+e$,
$f(y)= \sum_{i=0}^{n} \sum_{j=0}^{m} c_{ij} x^iy^j$,
where $a,b,e,c_{ij} \in k, c_{nm} \neq 0$.
If both $a$ and $b$ are non-zero, then we can define
$g(x):= (1/a)(x-y)$, $g(y):=(1/b)y$,
and get
$(gf)(x)= g(f(x))= g(ax+by+e)$
$=ag(x)+bg(y)+e= a(1/a)(x-y)+b(1/b)y+e$
$=x-y+y+e=x+e$.
$gf(y)= \sum_{i=0}^{N} \sum_{j=0}^{M} d_{ij} x^iy^j$,
for some $d_{ij} \in k, d_{NM} \neq 0$.
We will work with $gf$, show that $gf$ is invertible, 
and then obviously $f$ is invertible.

\textbf{First proof:} By a direct computation of the Jacobian.

\textbf{Second proof:} This proof is over $\mathbb{C}$, 
since we use C-M-W's theorem (however, C-M-W's theorem seems to be 
valid over any field of characteristic zero, or at least 
over $\mathbb{C}$ and $\mathbb{R}$,
hence also this proof).
$(gf)(x)= x+e$ has a Jacobian mate
(namely $y$). 
Write $(gf)(y)= Q_0+Q_1+\ldots+Q_r$,
where $Q_j$ is the $(1,1)$-homogeneous component of $(gf)(y)$
of $(1,1)$-degree $j$.
$k^* \ni \Jac(x+e, Q_0+Q_1+\ldots+Q_r)$
$= \Jac(x+e, Q_1)+ \ldots+ \Jac(x+e, Q_r)$.
From considerations of degree
(the $(1,1)$-degree of $\Jac(x+e, Q_j)$ equals $j-1$,
$1 \leq j \leq r$)
we have:
$\Jac(x+e, Q_1) \in k^*$,
$\Jac(x+e, Q_j)=0$ for $2 \leq j \leq r$.
Apply C-M-W's theorem to
$\Jac(x+e, Q_j)=0$ for $2 \leq j \leq r$,
and get that 
$Q_j \in k[x+e]=k[x]$ for $2 \leq j \leq r$.
So,
$Q_j = H_j(x)$, $H_j(t) \in k[t]$
for $2 \leq j \leq r$.
If $\Jac(x+e, Q_1) = a \in k^*$,
then $Q_1 = ay + H_1(x)$,
where $H_1(t) \in k[t]$.
Concluding that 
$(gf)(y)= Q_0+Q_1+\ldots+Q_r$
$=Q_0+ay+H_1(x)+H_2(x)+\ldots+H_r(x)$.
So,
$(gf)(y)= ay+H(x)$.
($H(x)=Q_0+H_1(x)+H_2(x)+\ldots+H_r(x)$ is a polynomial in $x$ over $k$).
Clearly such $gf$ is invertible.
If C-M-W's theorem is valid over any field of characteristic zero,
then so is this second proof.

\textbf{Third proof:} This proof is valid over any field of 
characteristic zero, since we use
~\cite[Lemma 10.2.4 (i)]{essen} which is valid over any field of characteristic zero.
Somewhat similar to the second proof,
use ~\cite[Lemma 10.2.4 (i)]{essen} 
(see also ~\cite[Proposition 2.1 (b)]{GGV})
instead of C-M-W's theorem.
\end{proof}

\subsection{First observation}
First we wish to prove the following proposition,
which was mentioned in the introduction:

\begin{proposition}[A special case of Wang's theorem]\label{prop wang 2}
Let $f: k[x,y] \to k[x,y]$ be a morphism 
that satisfies
$\Jac(f(x),f(y)) \in k^*$.
If the degree of $f(x)$ is $\leq 2$ and
the degree of $f(y)$ is $\leq 2$,
then $f$ is invertible. 
\end{proposition}

\begin{proof}
Write:
$f(x)=p_0+p_1+p_2$,
$f(y)=q_0+q_1+q_2$,
where $p_i,q_i$ is the $(1,1)$-homogeneous component
of $f(x),f(y)$ respectively, having $(1,1)$-degree $i$.
$k^* \ni \Jac(f(x),f(y))=\Jac(p_0+p_1+p_2,q_0+q_1+q_2)$.
{}From considerations of $(1,1)$-degrees,
$\Jac(p_2,q_2) = 0$,
hence from ~\cite[Lemma 10.2.4 (i)]{essen} or
~\cite[Proposition 2.1 (b)]{GGV} it follows that there 
exists a $(1,1)$-homogeneous element $R \in k[x,y]$,
such that:\begin{itemize}
\item [(1)] If $R$ is of degree $1$, then
$p_2= \lambda R^2$ and
$q_2= \mu R^2$ for some
$\lambda, \mu \in k^*$.
\item [(2)] If $R$ is of degree $2$, then
$p_2= \lambda R$ and
$q_2= \mu R$ for some
$\lambda, \mu \in k^*$
\end{itemize}
In each case it is not difficult to obtain that $f$
is invertible;
just take 
$(fg)(x)=f(g(x))=f(x-(\lambda/\mu)y)$
$=f(x)-(\lambda/\mu)f(y)$
$=p_0+p_1+p_2-(\lambda/\mu)(q_0+q_1+q_2)$
$=p_0+p_1+\lambda \tilde{R} -(\lambda/\mu)(q_0+q_1+\mu \tilde{R})$
$=p_0+p_1 -(\lambda/\mu)(q_0+q_1)$,
where $\tilde{R}= R^2$ for $(1)$,
$\tilde{R}= R$ for $(2)$,
and $g$ is the invertible morphism defined by:
$g(x):= x- (\lambda/\mu)y$
and $g(y):= y$.
(Notice that if $\mu=0$, then $f(y)$ is of degree $1$,
and then Lemma \ref{degree 1 case} immediately shows that
$f$ is invertible).
We can apply Lemma \ref{degree 1 case} to $fg$
($(fg)(x)$ is of degree $1$), get that $fg$ is invertible,
and then $f$ is invertible.
\end{proof}

Now we are finally in the position to bring:

\begin{theorem}[Our first observation]\label{1}
Let $f: \mathbb{C}[x,y] \to \mathbb{C}[x,y]$ be a morphism 
that satisfies
$\Jac(f(x),f(y)) \in \mathbb{C}^*$.
If the degree of $f(x)$ or the degree of $f(y)$ is $\leq 2$,
then $f$ is invertible. 
\end{theorem}

The proof uses Theorem \ref{p symmetric alpha} and its corollary \ref{p symmetric delta}
(Theorem \ref{p symmetric alpha} is based on the $\alpha$ 
restriction theorem \ref{alpha res thm}, 
which is based on Theorem 
\ref{alpha true} ~\cite[Proposition 4.1]{vered valqui}).
There is really nothing more than that,
only some quite easy calculations. 

\begin{proof}
For convenience denote:
$p:= f(x)$ and $q:= f(y)$.

\textbf{First option:}
The degree of $p$ is $1$, so Lemma \ref{degree 1 case} 
shows that $f$ is invertible.

Another argument is as follows; we bring it since it is in the spirit 
of the arguments which we will immediately bring in the second option:
Write $p=ax+by+c$,
for some $a,b,c \in \mathbb{C}$
with at least one of $a,b$ non-zero.

\textbf{I} If $b=0$, 
then $f(x)=p=ax+c$ is symmetric with respect to $\beta$.
By Corollary \ref{p symmetric delta}, $f$ is invertible.

\textbf{II} If $a=0$, 
then $f(x)=p=by+c$ is symmetric with respect to $\gamma$.
By Corollary \ref{p symmetric delta}, $f$ is invertible.

\textbf{III} If both $a$ and $b$ are non-zero,
then we can define the invertible morphism
$g(x)=bx$, $g(y)=ay$.
The morphism $gf$ has an invertible Jacobian 
(since the Jacobian of each of $f$ and $g$
is invertible, and by the Chain Rule) and
$(gf)(x)=g(f(x))=g(p)=g(ax+by+c)=
ag(x)+bg(y)+c=a(bx)+b(ay)+c=ab(x+y)+c$
is symmetric with respect to the exchange involution
$\alpha$.
By Theorem \ref{p symmetric alpha}, $gf$ is invertible.
Then clearly $f$ is invertible 
($f=g^{-1}(gf)$ is a composition of two invertible morphisms: 
$g^{-1}$ and $gf$).

\textbf{Second option:} The degree of $p$ is $2$,
so $p=ax^2+bxy+cy^2+dx+ey+r$,
for some $a,b,c,d,e,r \in \mathbb{C}$
with at least one of $a,b,c$ non-zero.

We consider three main cases: $a=c=0$;  
one of $a,c$ is zero and the other is non-zero;
both $a$ and $c$ are non-zero.
Each of the three main cases is further divided into sub-cases.

In each sub-case we get that $(g_2g_1f)(x)$  
or $(g_1f)(x)$ or $f(x)$ is symmetric with respect to 
some involution $\in C_{-1}$,
where $g_i$ are some invertible morphisms.
Then by Corollary \ref{p symmetric delta},
$\prod{g_i}f$ is invertible, hence $f$ is invertible.

\textbf{Case I} $a=c=0$ (hence $b \neq 0$). 

\textbf{I(1)} If $d=e=0$, then $p=bxy+r$ is already symmetric
with respect to $\alpha$.

\textbf{I(2)} If $e=0$, $d \neq 0$, then $p=bxy+dx+r$.
Since $b \neq 0$, we can define the following morphism
which is obviously invertible:
$g_1(x)=x$, $g_1(y)=y/b-d/b$.
Then 
$(g_1f)(x)=
bx(y/b-d/b)+dx+r=xy-dx+dx+r=xy+r$
is symmetric w.r.t. $\alpha$.

\textbf{I(3)} If $d=0$, $e \neq 0$, 
then $p=bxy+ey+r$.
Define the following invertible morphism:
$g_1(x)=x/b-e/b$, $g_1(y)=y$.
Then 
$(g_1f)(x)=
b(x/b-e/b)y+ey+r=xy-ey+ey+r=xy+r$
is symmetric w.r.t. $\alpha$.

\textbf{I(4)} If $d \neq 0$ and $e \neq 0$, 
then define the following invertible morphism:
$g_1(x)=ex$, $g_1(y)=dy$.
Then 
$(g_1f)(x)=
b(ex)(dy)+d(ex)+e(dy)+r=bdexy+de(x+y)+r$ 
is symmetric w.r.t. $\alpha$. 

\textbf{Case II} One of $a,c$ is zero and the other is non-zero,
without loss of generality assume $c=0$ and $a \neq 0$.

\textbf{II(1)} $b=0$: 
So $p=ax^2+dx+ey+r$.
\begin{itemize}
\item If $d=e=0$, then $p=ax^2+r$ is symmetric w.r.t. 
$\beta$ (or $\gamma$).
\item If $e=0$, then $p=ax^2+dx+r$ is symmetric w.r.t 
$\beta$.
\item If $d=0$, then $p=ax^2+ey+r$ is symmetric w.r.t. 
$\gamma$.
\item If $d \neq 0$ and $e \neq 0$, 
then define the following invertible morphism:
$g_1(x)=ex$, $g_1(y)=dy$.
Then
$(g_1f)(x)=ae^2x^2+de(x+y)+r$.
Take $g_2(x)=x$, $g_2(y)=-x+y$,
and get:
$(g_2g_1f)(x)=ae^2x^2+dey+r$
which is symmetric w.r.t. $\gamma$.
\end{itemize}

\textbf{II(2)} $b \neq 0$:
So $p=ax^2+bxy+dx+ey+r$.
Define the following invertible morphism:
$g_1(x)=x/\sqrt{a}-iy/\sqrt{a}$,
$g_1(y)=2\sqrt{a}iy/b$.
Then
$(g_1f)(x)=x^2+y^2+Dx+Ey+r$,
where
$D=d/\sqrt{a}$, $E=-di/\sqrt{a}+2e\sqrt{a}i/b$.
Take
$g_2(x)=x-D/2$, $g_2(y)=y-E/2$.
Then 
$(g_2g_1f)(x)=x^2+y^2-D^2/4-E^2/4+r$
is symmetric w.r.t. $\alpha$.

\textbf{Case III} Both $a$ and $c$ are non-zero.

\textbf{III(1)} $b=0$:
So $p=ax^2+cy^2+dx+ey+r$.
Since $a \neq 0$ and $c \neq 0$,
we can define:
$g_1(x)=x/\sqrt{a}$, $g_1(y)=y/\sqrt{c}$.
Then 
$(g_1f)(x)=x^2+y^2+dx/\sqrt{a}+ey/\sqrt{c}+r$.

Remark: $a,c \in \mathbb{C}$, so there exist $\sqrt{a},\sqrt{c}$.
This shows that the same proof will not necessarily work over $\mathbb{R}$;
we shall have to at least add the condition that both coefficients  
$a$ and $c$ are positive or both are negative;
if both are negative, we will have to consider $-p$ instead of $p$, 
show that $-f$ is invertible, hence $f$ is invertible.

Now define:
$g_2(x)=x-d/2\sqrt{a}$, $g_2(y)=y-e/2\sqrt{c}$.
Then
$(g_2g_1f)(x)=x^2+y^2-d^2/4a-e^2/4c+r$
is symmetric w.r.t. $\alpha$.

\textbf{III(2)} $b \neq 0$:
Again we can define:
$g_1(x)=x/\sqrt{a}$, $g_1(y)=y/\sqrt{c}$.
Then 
$(g_1f)(x)=x^2+Bxy+y^2+Dx+Ey+r$,
where 
$B=b/\sqrt{ac}$, $D=d/\sqrt{a}$, $E=e/\sqrt{c}$.
Now,
$x^2+Bxy+y^2+Dx+Ey+r=
(x+By/2)^2+ty^2+Dx+Ey+r$, 
where
$t=1-B^2/4$.
\begin{itemize}
\item If $t \neq 0$, then we can define
$g_2(x)=x-By/2\sqrt{t}$, $g_2(y)=y/\sqrt{t}$.
So
$(g_2g_1f)(x)=x^2+y^2+\tilde{D}x+\tilde{E}y+r$,
where
$\tilde{D}=D$ and 
$\tilde{E}=-DB/2\sqrt{t}+E/\sqrt{t}$,
which is of the form of \textbf{III(1)},
so $g_2g_1f$ is invertible.
\item If $t=0$, then $B=2$ or $B=-2$:
If $B=2$, then
$(g_1f)(x)=(x+y)^2+Dx+Ey+r$.
Take 
$g_2(x)=x+y$, $g_2(y)=-y$.
Then
$(g_2g_1f)(x)=x^2+Dx+(D-E)y+r$,
which is of the form of \textbf{II(1)},
so $g_2g_1f$ is invertible.
If $B=-2$, then
$(g_1f)(x)=(x-y)^2+Dx+Ey+r$.
Take 
$g_2(x)=x+y$, $g_2(y)=y$.
Then
$(g_2g_1f)(x)=x^2+Dx+(D+E)y+r$,
which is of the form of \textbf{II(1)},
so $g_2g_1f$ is invertible.
\end{itemize}
\end{proof}
\subsection{Second observation}
 
\begin{theorem}[Our second observation]\label{2}
Let $K$ denote $\mathbb{C}$ or $\mathbb{R}$.
Let $f: K[x,y] \to K[x,y]$
be a morphism that satisfies
$\Jac(f(x),f(y)) \in K^*$.
If one of the following two conditions is satisfied, 
then $f$ is invertible:
\begin{itemize}
\item [(1)] The $(0,1)$-degrees of all monomials 
in $f(x)$ are of the same parity.
\item [(2)] The $(1,0)$-degrees of all monomials 
in $f(x)$ are of the same parity.
\end{itemize}
\end{theorem}

\begin{proof}
If all monomials in $f(x)$ are of odd $(0,1)$-degrees
then $f(x)$ is skew-symmetric w.r.t. $\beta$.
If all monomials in $f(x)$ are of even $(0,1)$-degrees
then $f(x)$ is symmetric w.r.t. $\beta$.
If all monomials in $f(x)$ are of odd $(1,0)$-degrees
then $f(x)$ is skew-symmetric w.r.t. $\gamma$.
If all monomials in $f(x)$ are of even $(1,0)$-degrees
then $f(x)$ is symmetric w.r.t. $\gamma$.
In each case, Corollary \ref{p symmetric delta} shows that $f$ is invertible
\end{proof}
\section{Generalizations}
We suggest the following ideas how to generalize Theorems \ref{1} and \ref{2}.
We have not yet seriously checked each idea,
so maybe some of them will not work;
we hope that it will turn out that some of them will work.

\subsection{First idea} 
The following theorem is true, as a special case of theorems
$(1),(2),(4)$ of Remark \ref{other nice conditions}.
However, we wish to prove it in a way similar to the way we proved 
Theorem \ref{1}, namely, using Corollary \ref{p symmetric delta}.

\begin{theorem}\label{first idea}
Let $f:\mathbb{C}[x,y] \to \mathbb{C}[x,y]$
be a morphism 
that satisfies

$\Jac(f(x),f(y)) \in \mathbb{C}^*$.
If the degree of $f(x)$ or the degree of $f(y)$ is $\leq 3$,
then $f$ is invertible.
\end{theorem}

\begin{proof}[Sketch of proof]
If the degree of $f(x)$ is $\leq 2$, 
then $f$ is invertible by Theorem \ref{1}.
If the degree of $f(x)$ is $3$, then $f(x)$ is of the following form:
$f(x)=ax^3+bx^2y+cxy^2+dy^3+Ax^2+Bxy+Cy^2+Dx+Ey+R$,
where $a,b,c,d,A,B,C,D,E,R \in \mathbb{C}$
with at least one of $a,b,c,d$ non-zero.
We guess it is possible to find invertible morphisms
$g_1,\ldots,g_l$, such that $(\prod{g_i}f)(x)$ is 
symmetric or skew-symmetric w.r.t. some involution $\in C_{-1}$. 
If so, then by Corollary \ref{p symmetric delta}, 
$\prod{g_i}f$ is invertible, 
hence so is $f$.
\end{proof}

Actually, what is important in the proof of Theorem \ref{1}
and in the sketch of proof of Theorem \ref{first idea}
is not the degree of $f(x)$,
but the existence of invertible morphisms $g_1,\ldots,g_l$
such that $(\prod{g_i}f)(x)$ becomes 
symmetric or skew-symmetric w.r.t. some involution $\in C_{-1}$ 
(When the degree of $f(x)$ is $\leq 2$, 
it is not difficult to find such $g_i$'s,
as we have seen in the proof of Theorem \ref{1}.
When the degree of $f(x)$ is $3$, 
the situation becomes more complicated, 
and we have not yet found such $g_i$'s,
only in some special cases, 
one of them is brought in Example \ref{example}).
Therefore we make the following conjecture:

\begin{conjecture}\label{p symmetric}
Let $K$ denote $\mathbb{C}$ or $\mathbb{R}$.
Assume $A \in K[x,y]$ has a Jacobian mate.
Then there exist invertible morphisms $g_1,\ldots,g_l$
such that $\prod{g_i}A$ is 
symmetric or skew-symmetric w.r.t. some involution on $K[x,y]$
that has Jacobian $-1$.
\end{conjecture}

\begin{example}\label{example first idea}
Let $A=x+y^3$.
Clearly, $A$ is skew-symmetric w.r.t. $\epsilon$,
but the Jacobian of $\epsilon$ equals $1$, so $\epsilon$
is not relevant for Conjecture \ref{p symmetric}.
However, we can take $g_1(x):=x-y^3, g_1(y):=y$,
and get
$g_1A=g_1(x+y^3)=g_1(x)+g_1(y^3)=x-y^3+y^3=x$,
which is symmetric w.r.t. $\beta$
(and is skew-symmetric w.r.t. $\gamma$).
\end{example}

We suspect that Conjecture \ref{p symmetric} is true,
or at least the following weaker version of it is true, 
in which we allow the involution to have either Jacobian $-1$ or $1$:

\begin{conjecture}\label{p symmetric weak}
Let $K$ denote $\mathbb{C}$ or $\mathbb{R}$.
Assume $A \in \mathbb{C}[x,y]$ has a Jacobian mate.
Then there exist invertible morphisms $g_1,\ldots,g_l$
such that $\prod{g_i}A$ is 
symmetric or skew-symmetric w.r.t. some involution on $K[x,y]$. 
\end{conjecture}

Of course, what is nice in Conjecture \ref{p symmetric} 
is that a positive answer to it implies a positive answer to the 
two dimensional Jacobian Conjecture:

\begin{theorem}\label{if conj then JC}
If Conjecture \ref{p symmetric} is true,
then the two dimensional Jacobian Conjecture 
(over $\mathbb{C}$ or $\mathbb{R}$) is true.
\end{theorem}

\begin{proof}
Let $K$ denote $\mathbb{C}$ or $\mathbb{R}$.
Let $f: K[x,y] \to K[x,y]$
be a morphism that satisfies
$\Jac(f(x),f(y)) \in K^*$.
$f(x)$ has a Jacobian mate, so Conjecture \ref{p symmetric}
implies that there exist invertible morphisms $g_1,\ldots,g_l$
such that $\prod{g_i}f(x)$ is 
symmetric or skew-symmetric w.r.t. some involution on $K[x,y]$
that has Jacobian $-1$.
Apply Corollary \ref{p symmetric delta} to
$\prod{g_i}f$ and obtain that $\prod{g_i}f$ is invertible.
Then $f$ is invertible.
\end{proof}

If Conjecture \ref{p symmetric} is not true, but
its weak version Conjecture \ref{p symmetric weak} is true, 
we do not have an analogous result to Theorem \ref{if conj then JC} yet
(hopefully, we will have an analogous result to Theorem \ref{if conj then JC});
this is because we have not yet proved
the $\epsilon$ Jacobian Conjecture \ref{epsilon conj}.

Indeed, if the $\epsilon$ Jacobian Conjecture is true, 
then it is not difficult to obtain analogous results to
the $\alpha$ restriction theorem \ref{alpha res thm},
Theorem \ref{p symmetric alpha} and its corollary \ref{p symmetric delta},
with $\alpha$ replaced by $\epsilon$ and $\delta$ of 
Corollary \ref{p symmetric delta}, which is any involution that is conjugate to
$\alpha$, replaced by any involution that is conjugate to $\epsilon$.

\subsection{Second idea}
If the following ``proof" has no errors, 
then the following theorem is a generalization of Wang's theorem
~\cite[Theorem 62]{wang sep} 
(when $D=\mathbb{C}$ in Wang's theorem). 

\begin{theorem}\label{second idea}
Let
$f: \mathbb{C}[x_1,\ldots,x_n] \to \mathbb{C}[x_1,\ldots,x_n]$ 
be a morphism having an invertible Jacobian ($n \geq 2$).
If there exists $1 \leq i_0 \leq n$ such that the degree of 
$f(x_{i_0})$ is $\leq 2$,
then $f$ is invertible.
\end{theorem}

\begin{proof}
\textbf{First idea of a proof:}
$f(x_{i_0})$ is of degree $\leq 2$, then 
$f(x_{i_0})=ax_1^2+bx_1x_2+cx_2^2+ux_1+vx_2+w$, where 
$a,b,c \in \mathbb{C}$,
$u,v,w \in \mathbb{C}[x_3,\ldots,x_n]$.
So now we wish to work with $R[x_1,x_2]$,
where $R:=\mathbb{C}[x_3,\ldots,x_n]$.
Notice that taking $\sqrt{a}, \sqrt{c}$ is still possible,
since here also $a,c \in \mathbb{C}$.

Caution: Working over $R=\mathbb{C}[x_3,\ldots,x_n]$ 
instead of working over a field of characteristic zero
may cause problems:
\begin{itemize}
\item The positive answer to the $\alpha$ Jacobian Conjecture, 
Theorem \ref{alpha true} ~\cite[Proposition 4.1]{vered valqui},
is over a field of characteristic zero, 
but maybe it is still valid over $R$.
\item C-M-W's theorem is originally over $\mathbb{C}$; 
we have seen that probably a more general field than $\mathbb{C}$ can be taken
(Lemmas \ref{cmw}, \ref{lemma alg clo char 0} and \ref{lemma char 0}). 
We have not yet checked if C-M-W's theorem is still valid over an integral domain 
(of characteristic zero). 
Maybe a way to overcome this is to consider the field of fractions of $R$,
denote it $Q(R)$, 
apply C-M-W's theorem over $Q(R)$ and then consider $R \subset Q(R)$
(similary to what we have done in Lemma \ref{cmw}).
\end{itemize}

However, even if those problems can be solved, considering 
$R[x_1,x_2]$ may not help, since the Jacobian of $f(x_{i_0})$
and $f(x_j)$, $j \neq i_0$, with respect to $x_1,x_2$,
does not necessarily belong to $R^*= \mathbb{C}^*$.

\textbf{Second idea of a proof:}
To prove that the $\alpha$ Jacobian Conjecture has a positive answer,
and to prove the $\alpha$ restriction theorem,
both for $k[x_1,\ldots,x_n]$, $n > 2$.
Then to check how to generalize Theorem \ref{p symmetric alpha}
and its corollary \ref{p symmetric delta} to $n>2$;
we will have to find the correct generalization of C-M-W's theorem
(Theorem \ref{p symmetric alpha} relies on C-M-W's theorem).
Perhaps we will have to demand that for $n-1$ indices 
$\subset \{1,\ldots,n\}$
the associate $n-1$ $f(x_j)$'s are all symmetric w.r.t. 
some involution having Jacobian $-1$, 
in case demanding for only one index will not be enough.

Only after having all these generalizations, we can see if it is possible to prove
the current theorem; perhaps we will have to demand in the current theorem
that for $n-1$ indices 
$\subset \{1,\ldots,n\}$ the associate $n-1$ $f(x_j)$'s are all of degree $\leq 2$.

(Remark: If in the generalized Theorem \ref{p symmetric alpha},
demanding for only one index will be enough,
namely demanding that there exists $1 \leq i_0 \leq n$ 
such that $f(x_{i_0})$ is symmetric w.r.t. some involution 
having Jacobian $-1$, 
then we may also have a similar conjecture to Conjecture \ref{p symmetric}).
\end{proof}
\subsection{Third idea}
We hope that it is possible to generalize our results here
and show that cubic homogeneous polynomial maps 
(also called Yagzhev maps) are invertible, 
see Bass, Connell and Wright ~\cite{bass connell wright},
Yagzhev ~\cite{yagzhev}, ~\cite[Theorem 6.3.1]{essen}.
It may be less difficult to show that cubic linear polynomial maps 
(also called Druzkowski maps) are invertible, 
see ~\cite{dru}, ~\cite[Theorem 6.3.2]{essen}.

\begin{proposition}\label{druz}
When $n=2$, Druzkowski maps are invertible.
\end{proposition}

\begin{proof}
Let $d$ be a Druzkowski map.
It is easy to find invertible morphisms such that composing them 
an $d(x)$ yields a morphism which takes $x$ to
$\lambda x + \mu y^3$.
We have seen in Example \ref{example first idea}
that such a morphism is invertible.
\end{proof}

Proposition \ref{druz} is not enough, since we must prove that for 
\textit{all} $n \geq 2$, Druzkowski maps are invertible,
not just for $n=2$. 
However, we strongly suspect that it is possible to show,
with the help of the results brought in this paper, that 
cubic linear polynomial maps (Druzkowski maps) are invertible:

If the generalizations we suggested in the ``proof" of 
Theorem \ref{second idea} (second idea) are indeed possible,
namely, if it is possible to get a generalized version ($n>2$)
of the positive answer to the $\alpha$ Jacobian Conjecture,
the $\alpha$ restriction theorem, Theorem \ref{p symmetric alpha}
and its corollary \ref{p symmetric delta},
then we hope that it is not difficult to show
that Druzkowski maps are invertible;
this depends on what exactly a generalization of 
Corollary \ref{p symmetric delta} will be: \begin{itemize}
\item ``A good generalization": 
Invertibility of $f$ can be obtained by demanding
that there exists $1 \leq i_0 \leq n$ 
such that $f(x_{i_0})$ is symmetric w.r.t. some involution 
having Jacobian $-1$. 
Denote a given Druzkowski map by $d$.
Then it seems possible, in a similar way to what we have seen in the proof of 
Proposition \ref{druz}, 
to find invertible morphisms $g_1,\ldots,g_l$ such that
$(\prod{g_i}d)(x_1)$ is symmetric w.r.t. some involution having Jacobian $-1$.
(Of course, we can take any $(\prod{g_i}d)(x_i)$, not necessarily 
$(\prod{g_i}d)(x_1)$, since all $d(x_i)$ are of the same form).

\item ``Not as good as the above generalization": 
Invertibility of $f$ can be obtained by demanding
that there exist $n-1$ indices $\subset \{1,\ldots,n\}$
such that the associate $n-1$ $f(x_j)$'s are all symmetric w.r.t. some 
involution having Jacobian $-1$.
Now it seems less easy (but hopefully still possible) to find 
invertible morphisms $g_1,\ldots,g_l$ such that
for $n-1$ indices $j$,
the $n-1$ $(\prod{g_i}d)(x_j)$'s are symmetric w.r.t. some involution 
having Jacobian $-1$.
\end{itemize}

\subsection{Fourth idea:} We suggest to generalize Theorem \ref{2} to all $n > 2$. 
We wonder if one can find new proofs for known results
(for example ~\cite{essen hubbers}) based on Corollary \ref{p symmetric delta}.

\subsection{Fifth idea:} 
In the introduction we mentioned two great things in Wang's theorem.
The first thing seems more important than the second, 
at least from the point of view of trying to solve the Jacobian Conjecture, 
since it reminds us of the above mentioned Theorems of 
Bass, Connell and Wright ~\cite{bass connell wright},
Yagzhev ~\cite{yagzhev} and Druzkowski ~\cite{dru}. 
We wonder if one can adjust the proofs of those theorems to show that it 
suffices to prove the Jacobian Conjecture
for all $n \geq 2$ and all quadratic polynomial maps
(in order to do so, maybe one should consider a new notion of ``stable equivalence" 
which involves degree $2$ instead of degree $3$, 
see ~\cite[pages 119-124]{essen}).

\bibliographystyle{plain}

\end{document}